\documentclass[12pt]{amsart}

\usepackage{a4wide}
\usepackage{graphicx,eepic}
\usepackage{amsmath,amsfonts,amssymb, amscd}
\usepackage{times}
\usepackage{bm}

\newcommand{\be}{\beta}
\newcommand{\de}{\delta}
\newcommand{\al}{\alpha}
\newcommand{\e}{\varepsilon}

\newcommand{\si}{\sigma}
\newcommand{\Si}{\Sigma}

\newcommand{\m}{\mathfrak{m}}
\newcommand{\ov}{\overline}

\newcommand{\p}{\bm p}

\newtheorem{lemma}{Lemma}[section]

\newtheorem{prop}[lemma]{Proposition}
\newtheorem{thm}[lemma]{Theorem}
\newtheorem{cor}[lemma]{Corollary}

\theoremstyle{definition}
\newtheorem{Def}[lemma]{Definition}
\newtheorem{example}[lemma]{Example}

\theoremstyle{remark}
\newtheorem{rmk}[lemma]{Remark}

\numberwithin{equation}{section} \numberwithin{table}{section}

\title[Periodic unique beta-expansions]
{Periodic unique beta-expansions: \\
the Sharkovski\u{\i} ordering}

\author{Jean-Paul Allouche}

\author{Matthew Clarke}

\author{Nikita Sidorov}

\address{CNRS, LRI, UMR 8623, Universit\'e Paris-Sud,  B\^atiment 490,
F-91405 Orsay Cedex, France. Email: allouche@lri.fr}

\address{Mathematical Institute, University of Oxford,
24-29 St Giles', Oxford, OX1 3LB, United Kingdom. Email:
matthew.clarke@maths.ox.ac.uk}

\address{School of Mathematics, The University of Manchester,  Oxford Road, Manchester M13 9PL, United Kingdom. E-mail:
sidorov@manchester.ac.uk}

\date{\today}

\dedicatory{To O.~M.~Sharkovski\u{\i} on the occasion of his 70th birthday}


\begin{document}

\begin{abstract}
Let $\be\in(1,2)$. Each $x\in[0,\frac{1}{\be-1}]$ can be represented
in the form
\[
x=\sum_{k=1}^\infty \e_k\be^{-k},
\]
where $\e_k\in\{0,1\}$ for all $k$ (a $\be$-expansion of $x$). If
$\be>\frac{1+\sqrt5}{2}$, then, as is well known, there always exist
$x\in(0,\frac1{\be-1})$ which have a unique $\be$-expansion.

In the present paper we study (purely) periodic unique
$\be$-expansions and show that for each $n\ge2$ there exists
$\be_n\in[\frac{1+\sqrt5}{2},2)$ such that there are no unique
periodic $\be$-expansions of smallest period $n$ for $\be\le\be_n$
and at least one such expansion for $\be>\be_n$.

Furthermore, we prove that $\be_k<\be_m$ if and only if $k$ is less
than $m$ in the sense of the Sharkovski\u{\i} ordering. We give two
proofs of this result, one of which is independent, and the other
one links it to the dynamics of a family of trapezoidal maps.

\end{abstract}

\maketitle

\section{History of the problem and formulation of results}
\label{sec:intro}

This paper continues the line of research related to the
combinatorics of representations of real numbers in non-integer
bases (\cite{EJK, GS, KV, S}).

Let $\be\in(1,2)$ be our parameter and let $x\in
I_\be:=[0,\frac{1}{\be-1}]$. Then $x$ has at least one
representation of the form
\begin{equation}\label{eq:beta-exp}
x=\pi_\be(\e_1,\e_2,\dots):=\sum_{k=1}^\infty \e_k\be^{-k},\quad
\e_k\in\{0,1\},
\end{equation}
(use, e.g., the greedy algorithm) which we call {\em a
$\be$-expansion of $x$} and write $x\sim(\e_1,\e_2,\dots)_\be$.

Let us recall some key results regarding $\be$-expansions. Firstly,
if $1<\be<G:=\frac{1+\sqrt5}{2}$, then each $x\in
\bigl(0,\frac1{\be-1}\bigr)$ has a continuum of $\be$-expansions
\cite{EJK}. On the other hand, for any $\be>G$, there exist
infinitely many $x$ which have a unique $\be$-expansion (see
\cite{DarKat, GS}), although almost all $x\in I_\be$ still have a
continuum of $\be$-expansions \cite{S}.

More specifically, put $x\sim(010101\dots)_\be=\frac1{\be^2-1}$.
Then both $x$ and $\be x$ have a unique $\be$-expansion \cite{GS}.

Let $X_\be$ denote the set of $x$ which have a unique
$\be$-expansion (these numbers are sometimes called univoque numbers
with respect to $\be$, see \cite{DK} for example). Denote by
\[
(\m_k)_{k=0}^\infty= 0110\,\, 1001\,\, 0110 \,\, 1001\dots
\]
the Thue-Morse sequence, i.e., the fixed point of the morphism
$0\to01, 1\to10$ (see, e.g., the survey paper \cite{AS} for the many
wonderful properties of this famous sequence). Let now
$\be_{KL}\approx1.78723$ denote the {\em Komornik-Loreti constant},
i.e., the unique positive solution of the equation
\[
\sum_{k=1}^\infty \m_k x^{-k}=1.
\]
This constant proves to be the smallest $\be$ such that $1\in X_\be$
-- see \cite{KL}. Note that in \cite{AC} it was shown that
$\be_{KL}$ is transcendental.

The main result of \cite{GS} asserts that the set $X_\be$ is

\begin{enumerate}
\item infinite countable if $G<\be<\be_{KL}$;
\item a continuum of zero Hausdorff dimension if $\be=\be_{KL}$; and
\item a continuum of Hausdorff dimension strictly between 0 and 1 if
$\be_{KL}<\be<2$.
\end{enumerate}

Notice that for $\be>\be_{KL}$ the set $X_\be$ is not necessarily a
Cantor set and may have a somewhat complicated topology -- see
\cite{KV} for more detail and also for the case of an arbitrary
$\be>1$.

Let $\Si_\be$ denote the set of all 0-1 sequences which are unique
$\be$-expansions, and $\si_\be$ denote the shift on $\Si_\be$, i.e.,
$\si_\be(\e_1,\e_2,\e_3,\dots)=(\e_2,\e_3,\dots)$. Let
$\pi_\be:\Si_\be\to X_\be$ be the projection given by
(\ref{eq:beta-exp})). It is obvious that $\pi_\be$ is a bijection.

Since $\pi_\be(\{\e:\e_1=0\})\subset\bigl[0,\frac1{\be}\bigr]$ and
$\pi_\be(\{\e:\e_1=1\})\subset
\bigl[\frac1{\be(\be-1)},\frac1{\be-1}\bigr]$, the set $X_\be$ has
an empty intersection with the middle interval
$\bigl[\frac1{\be},\frac1{\be(\be-1)}\bigr]$. Thus, we have the
following commutative diagram:

\[
\begin{CD}
\Si_\be @>\si_\be>> \Si_\be \\
@VV{\pi_\be}V  @VV{\pi_\be}V \\
X_\be @>{F_\be}>> X_\be
\end{CD}
\]
where
\[
F_\be(x)=\begin{cases} \be x,& 0\le x < \frac1{\be},\\
\be x -1, & \frac1{\be(\be-1)} < x \le \frac1{\be-1}.
\end{cases}
\]
To state the main theorem of the present paper, we recall that the
{\em Sharkovski\u{\i} ordering} on the natural numbers is as
follows:

\medskip

\begin{tabular}{cccccccccccccc}

& $3$ & $\rhd$ & $5$ & $\rhd$ & $7$ & $\rhd$  & $\cdots $ & $\rhd$ &
$2m+1$
& $\rhd$ & $\cdots$\\

$\rhd$ &$2\cdot 3$ & $\rhd$ & $2\cdot 5$ & $\rhd$ & $2\cdot 7$ &
$\rhd$ &$ \cdots $ & $\rhd$ & $2\cdot (2m+1)$ & $\rhd$ & $\cdots$\\

$\rhd$ &$4\cdot 3$ & $\rhd$ & $4\cdot 5$ & $\rhd$ & $4\cdot 7$ &
$\rhd$ & $ \cdots $ & $\rhd$ & $4\cdot (2m+1)$ & $\rhd$ & $\cdots$\\

&$\vdots$&&$\vdots$&&$\vdots$&&$$&&$\vdots$\\

$\rhd$ &$2^n\cdot 3$ & $\rhd$ & $2^n\cdot 5$ & $\rhd$ & $2^n\cdot 7$
& $\rhd$  &$ \cdots $ & $\rhd$ & $2^n\cdot (2m+1)$ & $\rhd$ &
$\cdots$\\

&$\vdots$&&$\vdots$&&$\vdots$&&$$&&$\vdots$\\

  & && $\cdots$ &$\rhd$ &$8$ &$\rhd$ &$4$ &$\rhd$ &$2$ &$\rhd$& $1$,\\

\end{tabular}

\medskip\noindent
where the relation $a \; \rhd \; b$ indicates that $a$
comes before $b$ in the ordering.

\begin{rmk}
This is a complete ordering on $\mathbb{N} := \{1, 2, 3, \ldots \}$
since $(n,m) \mapsto 2^n(2m+1)$ is a bijection
$(\mathbb{N}\cup\{0\})^2 \rightarrow \mathbb{N}$.
\end{rmk}

\begin{thm}
\emph{\textbf{(Sharkovski\u{\i}'s Theorem)}} Let $f$ be a continuous
map of the real line.   If \;$k \rhd l$ in Sharkovki\u{\i}'s
ordering and if $f$ has a point of smallest period $k$, then $f$
also has a point of smallest period $l$.
\end{thm}

This was originally proved in \cite{SH}, see also \cite{RD}.

Now we are ready to state the main theorem of the present paper. Put
\begin{align*}
U_n&=\{\be\in(1,2) : \Si_\be\,\,\text{contains a periodic sequence of smallest period $n$} \}.\\
&=\{\be\in(1,2) : F_\be: X_\be\to X_\be\ \ \text{has an $n$-cycle}.\}
\end{align*}
(By the result quoted above, $U_2=(G,2)$, for instance.)

\begin{thm}\label{thm:main}
There exist real numbers $\beta_n$ in $(1, 2)$ such that
$U_n=(\be_n,2)$ for any $n\ge2$. Furthermore, $\be_n<\be_m$ if and
only if $n\triangleleft m$ in the sense of the Sharkovski\u{\i}
ordering.
\end{thm}

\begin{rmk}
We are going to give explicit formulae for the $\be_n$ via the
fragments of the Thue-Morse sequence as well as the first $n$-cycle
to appear -- see Proposition~\ref{prop:explicit} below. Note that
for $n=2^k$ this result is essentially contained in \cite{GS}, where
it is also shown that $\be_{2^k}\nearrow\be_{KL}$.
\end{rmk}

We will make use {\it inter alia} of results in combinatorics on
words. In particular a set of binary sequences studied in
\cite{A-bordeaux} and denoted by $\Gamma$ (see (\ref{eq:gamma}))
will play a r\^ole in some proofs. In this respect one could compare
the combinatorial part of the main result of \cite{GS} with the
``Th\'eor\`eme'' and ``Corollaire'' on page~37 of \cite{A-etat}.

\medskip

The structure of this paper is as follows: the next section will be
devoted to the proof of Theorem~\ref{thm:main}. In
Section~\ref{sec:extension} we discuss possible links of this claim
with the classical theory of one-dimensional continuous maps. (Note
that Reference \cite{AC2} showed a link between kneading sequences
of unimodal continuous maps and unique $\beta$-expansions of $1$.)

Our central result of that section is in the negative direction: we
show that if $\be>\be_4$, then any continuous extension $S_\be$ of
$F_\be:X_\be\to X_\be$  has a $k$-cycle for any $k\ge2$ provided
$S_\be$ is monotonic on $[0,1/\be]$.

Section~\ref{sec:trapezoidal} is devoted to a different proof of our
Theorem~\ref{thm:main} via the classical Sharkovski\u{\i} theorem
applied to a family of trapezoidal maps, and  -- again -- some
combinatorics on words and properties of the set $\Gamma$.

\section{Proof of the main theorem}\label{sec:proof}

\begin{Def}

Put $\tau_\be x=\be x\bmod 1$, and $\e_n=[\be\tau_\be^{n-1}(x)]$ for
$x\in[0,1]$ and $n\ge1$. Then $(\e_1 \e_2 \dots)$ is called the {\em
greedy expansion} of $x$ in base~$\be$. Put $\Sigma=\{0,1\}^{\mathbb
N}$ and let $d(\be) \in \Sigma$ denote the greedy expansion of $1$
in base $\be$. If $d(\be)$ is finite, i.e., of the form $d(\be) =
\e_1 \dots \e_{n-1}\, 10^\infty$, then $d'(\be) :=
(\e_1\dots\e_{n-1}\,0)^{\infty}$ - the \emph{quasi-greedy expansion
of \, $1$}.

We say that the sequence (resp.~the finite word) $\e = \e_1 \e_2
\ldots$ is {\em lexicographically less} than the sequence (resp.~the
finite word with same length) $\e'$ if $\e_k<\e'_k$ for the least
$k$ such that $\e_k\neq\e_k'$. Notation: $\e\prec\e'$. We write $\e
\preceq \e'$ if either $\e\prec\e'$ or $\e = \e'$.

\end{Def}

We will use the following simple remark.

\begin{rmk}

The relation $\preceq$ is a total order on the set of infinite
sequences (resp.~on the set of words of given length). Further
let $\e$ and $\e'$ be two infinite sequences. Let $u$ and $u'$
be their respective prefixes of length say $\ell$, then $u \prec u'$
implies $\e \prec \e'$, and $\e \preceq \e'$ implies $u \preceq u'$.

\end{rmk}

We need the following auxiliary results on greedy and quasi-greedy
expansions:

\begin{lemma} \label{lemmas}

\begin{enumerate}
\item If $d'(\be)$ is defined, then it is also an expansion of \, $1$
in base~$\be$.

\item The equation
$1 = \sum_{i=1}^{\infty}\e_j x^{-j}$, for some fixed $\e =
\e_1\e_2\e_3\dots\in \Si$, with at least one $0$ and at
least two $1$s, always has a unique solution $\be\in(1,2)$.

\item \emph{\textbf{(Monotonicity)}}
Let $\be, \tilde{\be} \in (1,2)$. Then $\be > \tilde{\be}$ if, and
only if $d(\be) \succ d(\tilde{\be})$.

\item Let $\be, \tilde{\be} \in (1,2)$
and assume that $d'(\be)$ and $d'(\tilde{\be})$ are both defined and
of the same smallest period.  Then $\be > \tilde{\be}$ if, and only
if $d'(\be) \succ d'(\tilde{\be})$.

\item Assume $\e\in\Sigma$ satisfies \[ \sigma^j \e  \left\{
\begin{array}{ll}
         \prec \e & \mbox{if $j\not\equiv 0 \bmod n$};\\
         = \e & \mbox{if $j\equiv 0 \bmod n$}. \end{array} \right.\]
\noindent Then there exists $\be\in(1,2)$ such that $d'(\be)$ is
defined and equals $\e$.
\end{enumerate}

\end{lemma}
\begin{proof}
(1) is trivial exercise, while (2) follows from letting
$f(x):=\sum_{i=1}^{\infty}\frac{\varepsilon_i}{x^i} - 1$ and
observing that $f(1)>0$, $f(2)<0$, and $f'(x)<0$, $\forall
x\in(1,2)$; (3) is proved in \cite{Pa} and (4) is an easy
consequence of it.

(5) follows from the fact that necessarily $\e_{pn}=0$ for all
$p\in {\mathbb{N}}$ (otherwise the condition in question is not satisfied),
whence $\e_1 \dots \e_{n-1} 10^\infty$ is the greedy expansion of 1 in base~$\be$.
\end{proof}

\begin{thm} \label{parrycor} \emph{\textbf{(Parry \cite{Pa})}}
Let $\e \in \Si$. Then $\e = d(\be)$ for some  $\be \in (1,2)$ if,
and only if, $\si^j \e \prec \e$, $\forall j \ge1$.
\end{thm}

This is essentially proved in \cite{Pa}. The following auxiliary
lemmas will be needed later on.

\begin{lemma} \label{noroom}
There exist no $\be, \tilde\be \in(1,2)$ such that $d'(\be) \prec
d(\tilde\be) \prec d(\be)$.
\end{lemma}

\begin{proof}
Assume $d'(\be) := (\e_1\dots\e_{n-1}0)^{\infty}$ is defined and let
$d(\tilde\be) := \de_1\de_2\de_3\dots$ and suppose $d'(\be) \prec
d(\tilde\be) \prec d(\be)$. This immediately forces
\[
\de_1\de_2\de_3\dots\de_n = \e_1\e_2\e_3\dots\e_{n-1}0.
\]
By Theorem \ref{parrycor}, $\de_{n+1}\de_{n+2}\dots =  \sigma^n
d(\tilde\be) \prec d(\tilde\be) = \e_1 \dots \e_{n-1}0
\de_{n+1}\dots.$ Hence
\[
\de_{n+1} \dots \de_{2n} \preceq \e_1 \dots \e_{n-1}0.
\]
On the other hand $(\e_1 \dots \e_{n-1} 0 )^{\infty} = d'(\be) \prec
d(\tilde\be)$ implies that
\[
\de_{n+1} \dots \de_{2n} \succeq \e_1 \dots \e_{n-1}0,
\]
and hence
\[
\de_{n+1} \dots \de_{2n} = \e_1 \dots \e_{n-1}0.
\]
By repeating this process we see that we are forced into the
spurious conclusion that $d(\tilde\be) = d'(\be)$.
\end{proof}

Put
\[
\mathcal{A_{\be}} = X_{\be} \cap \left(\frac{2-\be}{\be-1},1\right).
\]
It is clear that $\mathcal A_\be$ is invariant under $F_\be$ and
moreover, it is an attractor for $F_\be$ (see \cite{GS} for more
detail).

Let $\ov\e$ denote the mirror image of $\e$, i.e., $(\ov\e)_n=1-\e_n$.

\begin{lemma} \cite{GS}
\label{crit}
Let $\be\in(1,2)$. Then
\begin{equation}\label{eq:uniq}
\Si_\be^{\mathcal A_\be}:=\pi_\be^{-1}(\mathcal A_\be)=
\left\{\e\in\Sigma\; : \; \ov{d(\be)} \prec \sigma^j \e \prec
d(\be),\quad \forall \, j \ge 0 \right\}
\end{equation}
if $d'(\be)$ is not defined. If it is, replace $d(\be)$ and
$\ov{d(\be)}$ with $d'(\be)$ and $\ov{d'(\be)}$ respectively in the
above.
\end{lemma}

\begin{rmk} In \cite{GS} it was shown that if
$\e\in\Sigma_\be$ is periodic, then $\pi_\be(\e)\in\mathcal A_\be$.
\end{rmk}

Put
\begin{equation}\label{eq:gamma}
\Gamma := \{\e\in\Si : \ov\e \preceq \si^k\e \preceq\e,\, \forall k\ge0\}.
\end{equation}
It is obvious that all sequences in $\Gamma$ begin with 1, and
furthermore, if $\e\in\Gamma$ begins with 10, then $\e=(10)^\infty$.
This set has been introduced and studied in detail in
\cite{A-bordeaux, A-etat}, as well as the sets $\Gamma_{\eta} :=
\left\{\e\in\Sigma\; : \; \ov{\eta} \preceq \sigma^j \e \preceq
\eta, \ \forall \, j \ge 0 \right\}$ for $\eta  \in \Gamma$.
(Actually the sequences in the set $\Gamma$ studied in
\cite{A-bordeaux, A-etat} satisfy the extra condition that they
begin with $11$, which only excludes from the present set $\Gamma$
the sequence $(10)^{\infty}$.)

\medskip

Define
\[
V_n := \{\be\in(1,2)\ :\ d'(\be) \mbox{ exists, has smallest
period~$n$, and } \\ d'(\be)\in\Gamma\},
\]
and let $\be_n := \inf U_n$ and $ \gamma_n := \min V_n$. (Provided
$U_n, V_n \neq\emptyset.$)

\begin{prop} \label{main}
Let $n\ge 2$. Then $ U_n, V_n \not= \emptyset$ and $\be_n =
\gamma_n$.
\end{prop}

\begin{proof}
For any $n\ge 2$, there always exists a $\be\in(1,2)$ such that
$d'(\be) = (1\dots 10)^{\infty}$ (of smallest period $n$) by Lemma
\ref{lemmas}, (5) so $V_n \not= \emptyset$. Notice also that $V_n$
is finite, so $\min V_n$ is well defined.

Let $\be\in (1,2)$, $\be > \gamma_n$. Then since $\gamma_n\in V_n$,
$\sigma^j d'(\gamma_n) \preceq d'(\gamma_n) \prec d(\gamma_n) \prec
d(\be)$ (using monotonicity), and similarly $\sigma^j
\ov{d'(\gamma_n)} \preceq d'(\gamma_n) \prec d(\gamma_n) \prec
d(\be)$, $\forall j \ge 0$. So we have $\ov{d(\be)} \prec \sigma^j
d'(\gamma_n) \prec d(\be), \; $ $\forall j \ge 0$ whence
$d'(\gamma_n) \in \Sigma_{\be}^{\mathcal{A_{\be}}} \subset
\Sigma_{\be}$ by Lemma \ref{crit} provided $d'(\be)$ does not exist.
(In the event that $d'(\be)$ does exist, this still holds via
$\ov{d'(\be)} \prec \sigma^j d'(\gamma_n) \prec d'(\be), \; $
$\forall j \ge 0$. Indeed if this were false then $d'(\be) \preceq
\sigma^j d'(\gamma_n) \preceq d'(\gamma_n) \prec d(\gamma_n) \prec
d(\be)$, but this is rendered absurd by Lemma \ref{noroom}.) Hence
$d'(\gamma_n) \in \Sigma_{\be}$, which implies $\be \in U_n$,
because $d'(\gamma_n)$ is periodic with smallest period $n$, and
therefore $\gamma_n \ge \be_n$, since $\be$ was arbitrary. Moreover,
we now know that $U_n \neq\emptyset$.

We complete the proof by showing $\gamma_n \le \be_n$. Let $\be \in
U_n$, $\be > \be_n$. So there exists $\e\in\Sigma_{\be}$ which is
periodic with smallest period $n$. Moreover, because $\e$ is
periodic, it must represent some number (in base $\be$) belonging to
$\mathcal{A_{\be}}$. I.e., $\e\in\Sigma_{\be}^{\mathcal{A_{\be}}}$.
Hence by Lemma \ref{crit}, $\ov{d(\be)} \prec \sigma^j \e \prec
d(\be),\quad \forall \, j \ge 0$. (If $d'(\be)$ exists, use the fact
that $d'(\be) \prec d(\be)$.)

Put $a = \max \{ \sigma^j \e , \sigma^j \ov{\e} \;\mid \; 0\le j \le
n-1 \}.$ So $\ov{a} \preceq \sigma^j a \preceq a$, $\forall j \ge
0$, i.e., $a\in\Gamma$. Moreover, since $a$ is periodic with
smallest period $n$, we must have, \[ \sigma^j a  \left\{
\begin{array}{ll}
         \prec a & \mbox{if $j\not\equiv 0$ mod $n$};\\
         = a & \mbox{if $j\equiv 0$ mod $n$}. \end{array} \right.\]

\noindent Hence $a = d'(\tilde{\be})$, for some $\tilde{\be}$ (by
Lemma \ref{lemmas}, (5)), and so $\tilde{\be}\in V_n$.

Finally, since $d'(\tilde{\be}) = a \prec d(\be),$ we must have
$\tilde{\be} \le \be$ by Lemma \ref{noroom}. I.e., for all $\be\in
U_n$, $\be> \be_n$, $\exists \; \tilde{\be}\in V_n$ such that
$\tilde{\be} \le \be$. Hence $\gamma_n \le \be_n$, as claimed.
\end{proof}

\begin{cor} \label{maincor}
$U_n = (\be_n,2)$.
\end{cor}

\begin{proof}
In the above it was shown that, if $\be\in (1,2)$ with $\be >
\gamma_n$, then $\be\in U_n$. Moreover, it was shown that $\gamma_n
= \be_n = \inf U_n$, hence $U_n = (\be_n,2)$, or $[\be_n,2)$. We now
discount the second case. Let $\be\in U_n$. So there exists a
periodic sequence, with smallest period $n$, $\e\in\Si$ such that
$\pi_{\be}(\e)\in X_{\be}.$ Moreover the fact that $\e$ is periodic
ensures that $\pi_{\be}(\e)\in\mathcal{A_{\be}}$ (see \cite{GS}).
Hence
\[
\ov{d(\be)} \prec \si^j\e \prec d(\be), \quad \forall j \in [0, n-1]
\]
by Lemma~\ref{crit}.

Putting $d(\be):=d_1d_2d_3\dots$, let $k$ be the smallest number such that
\[
\ov{d_1d_2d_3\dots d_k}1^\infty \prec \si^j\e \prec d_1d_2d_3\dots
d_k 0^\infty, \quad \forall j \in [0, n-1].
\]

Now by Theorem \ref{parrycor}, there exists $\be'\in(1,2)$ such that
$d(\be')=d_1d_2d_3\dots d_k 0^\infty$, whence
\[
\ov{d(\be')} \prec \si^j\e \prec d(\be'), \quad \forall j \in [0, n-1],
\]
i.e., $\be' \in U_n$ by Lemma~\ref{crit}. Now note that
$d(\be')\prec d(\be)$ implies $\be'<\be$ by monotonicity.
\end{proof}

We now introduce a result that allows us to make the connection to
the Sharkovski\u{\i} ordering.

\begin{prop}\label{shark-like}
If $a_k$ denote the lexicographically least sequence in $\Gamma$ of
smallest period $k$, then

\begin{tabular}{cccccccccccccc}

&$a_3$ & $\succ$ & $a_5$ & $\succ$ & $a_7$ & $\succ$  & $\cdots $ &
$\succ$ & $a_{2m+1}$ & $\succ$ & $\cdots$\\

$\succ$ &$a_{2\cdot 3}$ & $\succ$ & $a_{2\cdot 5}$ & $\succ$ &
$a_{2\cdot 7}$ & $\succ$ &$ \cdots $ & $\succ$ & $a_{2\cdot (2m+1)}$
& $\succ$ & $\cdots$\\

$\succ$ &$a_{4\cdot 3}$ & $\succ$ & $a_{4\cdot 5}$ & $\succ$ &
$a_{4\cdot 7}$ & $\succ$ & $ \cdots $ & $\succ$ & $a_{4\cdot
(2m+1)}$ & $\succ$ & $\cdots$\\

&$\vdots$&&$\vdots$&&$\vdots$&&$$&&$\vdots$\\

$\succ$ &$a_{2^n\cdot 3}$ & $\succ$ & $a_{2^n\cdot 5}$ & $\succ$ &
$a_{2^n\cdot 7}$ & $\succ$  &$ \cdots $ & $\succ$ & $a_{2^n\cdot
(2m+1)}$ & $\succ$ & $\cdots$\\

&$\vdots$&&$\vdots$&&$\vdots$&&$$&&$\vdots$\\

& && && $\cdots$ &$\succ$ &$a_8$ &$\succ$ &$a_4$ &$\succ$ &$a_2$.\\

\end{tabular}
\end{prop}

\begin{proof}
This proposition is a consequence of Proposition~\ref{smallest} and
Lemma~\ref{elementary-gamma}~(d) and (f) -- see below.
\end{proof}

Now by Lemma \ref{lemmas} (5), $a_k=d'(\be)$ for some $\be\in(1,2)$.
Furthermore, by Lemma~\ref{lemmas} (4), $\be=\be_k$. Hence, invoking
monotonicity allows us to apply the required ordering on the
$\be_k$. This reduces the proof of Theorem~\ref{thm:main} to proving
Proposition~\ref{shark-like}.

To calculate $\be_k$ explicitly, given $a_k$, we need to solve
$a_k=d'(\be)$ for $\be$. I.e., if $a_k=\al_1^{(k)}\al_2^{(k)}\dots$
then $\be_k$ is the unique root in $(1,2)$ of the polynomial
\[
x^k - \al_1^{(k)}x^{k-1} - \al_2^{(k)}x^{k-2} - \cdots -
\al_{k-1}^{(k)}x - 1.
\]

\bigskip

In order to prove Proposition~\ref{shark-like},  we will give a
construction of the sequences $(a_k)$. This construction was
suggested in \cite{A-etat} (see also \cite{A-bordeaux}). For the
sake of completeness, we give here a self-contained proof extracted
from these two references. We begin with a definition and two
lemmas.

\begin{Def}
We denote by $\mu$ the map defined on $\Sigma$ by:
if $\varepsilon = (\varepsilon_n)_{n \geq 1}$ belongs to $\Sigma$, then
$$
(\mu(\varepsilon))_1 := 1 \ \mbox{\rm and} \ \forall n \geq 1, \
\left\{
\begin{array}{lll}
(\mu(\varepsilon))_{2n} &:=& \varepsilon_n \\
(\mu(\varepsilon))_{2n+1} &:=& 1 - \varepsilon_n. \\
\end{array}
\right.
$$
We denote by $L := (\m_n)_{n \geq 1}$  the sequence obtained by
shifting the (complete) Thue-Morse sequence $(\m_n)_{n \geq 0}$.
\end{Def}

\begin{lemma}\label{elementary-gamma}
The following properties of the map $\mu$ and of the sequence $L$
hold true:

\begin{itemize}
\item[(a)] For all $k \geq 0$, $\sigma^{2k+1} \mu = \sigma \mu \ \sigma^k$.

\item[(b)] For any sequence $\varepsilon$ in $\Sigma$ we have
          $\sigma(\overline{\mu(\varepsilon)}) =
          \sigma(\mu(\overline{\varepsilon})) =
          \overline{\sigma(\mu(\varepsilon))}$.

\item[(c)] Let $\varepsilon$ be a sequence in $\Sigma$. If $\varepsilon$ is
          periodic with smallest period $T > 0$ and $\varepsilon_T = 0$, then
          $\mu(\varepsilon)$ is periodic with smallest period $2T$.
          In particular, if $\varepsilon$ is a sequence in $\Sigma$ such that
          $\mu(\varepsilon)$ is periodic with smallest period $U$, then $U$
          cannot be odd, hence $U$ is even, say $U=2T$, and $\varepsilon$ is
          periodic with smallest period $T$ and satisfies $\varepsilon_T = 0$.

\item[(d)] If the sequences $\varepsilon = (\varepsilon_n)_{n \geq 1}$
          and $\varepsilon' = (\varepsilon'_n)_{n \geq 1}$ in $\Sigma$
          satisfy $\varepsilon \prec \varepsilon'$, then
          $\mu(\varepsilon) \prec \mu(\varepsilon')$ and
          $\sigma(\mu(\varepsilon)) \prec \sigma(\mu(\varepsilon'))$.

\item[(e)] The sequence $L$ is a fixed point of the map $\mu$. Also, for
           any sequence $\varepsilon := (\varepsilon_n)_{n \geq 1}$, we have
           that $\mu^{\infty}(\varepsilon) := \lim_{n \to \infty}\mu^n(\varepsilon)$
           exists and $\mu^{\infty}(\varepsilon) = L$.

\item[(f)] Let $\varepsilon := (\varepsilon_n)_{n \geq 1}$ be a sequence in $\Sigma$.
          If $\varepsilon \prec L$, then $\varepsilon \prec \mu(\varepsilon)$.
          If $\varepsilon \succ L$, then $\varepsilon \succ \mu(\varepsilon)$.
\end{itemize}
\end{lemma}

\begin{proof}
The proofs of (a--d) are easy and left to the reader.

\medskip

Let us prove (e). It is straightforward that the sequence
$0(\mu(\varepsilon))_{n \geq 1}$ is the image of the sequence
$0(\varepsilon_n)_{n \geq 1}$ under the morphism $\varphi$ defined
by $\varphi(0) = 01$, $\varphi(1) = 10$, introduced above. In
particular, the sequence $L = (\m_n)_{n \geq 1}$ is a fixed point of
$\mu$, since the complete Thue-Morse sequence $(\m_n)_{n \geq 0}$ is
a fixed point of the morphism $\varphi$. Also we have, for any
sequence $\varepsilon$ in $\Sigma$ and any $k \ge 0$, the equality
$0\mu^k(\varepsilon) = \varphi^k(0\varepsilon)$. Hence
$0\mu^{\infty}(\varepsilon)$ exists and is equal to
$\varphi^{\infty}(0\varepsilon)$ which is precisely the (complete)
Thue-Morse sequence $(\m_n)_{n \geq 0}$.

\medskip

In order to prove (f), note that, if $\varepsilon \prec
\mu(\varepsilon)$, then, using (d), $\mu(\varepsilon) \prec
\mu(\mu(\varepsilon))$. Hence $\varepsilon \prec \mu(\varepsilon)
\prec \mu^2(\varepsilon)$, and by induction $\varepsilon \prec
\mu(\varepsilon) \prec \mu^k(\varepsilon)$ for all $k \geq 2$.
Letting $k$ tend to infinity and using (e) gives that $\varepsilon
\prec \mu(\varepsilon) \preceq L$, hence $\varepsilon \prec L$.
Reversing the inequalities show that $\varepsilon \succ
\mu(\varepsilon)$ implies that $ \varepsilon \succ L$, which proves
(f).
\end{proof}

\begin{lemma}\label{other-gamma}
The set $\Gamma$ has the following properties:
\begin{itemize}
\item[(a)] Let $\varepsilon$ be a sequence in $\Gamma$. Suppose that there
          exists $d \geq 1$ such that
$$
\varepsilon_{d+1} \ \varepsilon_{d+2} \cdots \varepsilon_{2d} =
\overline{\varepsilon_1 \ \varepsilon_2 \cdots \varepsilon_d}
$$
then the sequence $\varepsilon$ is periodic of period $2d$, i.e., we have
$$
\varepsilon = (\varepsilon_1 \ \varepsilon_2 \cdots \varepsilon_d \
\overline{\varepsilon_1 \ \varepsilon_2 \cdots \varepsilon_d})^{\infty}.
$$

\item[(b)] Let $\varepsilon$ be a sequence in $\Sigma$. Then $\varepsilon$
           belongs to $\Gamma$ if and only if $\varepsilon_1 = 1$ and
           $\mu(\varepsilon)$ belongs to $\Gamma$.

\item[(c)] If $\varepsilon$ is a sequence in $\Gamma$ such that
           $\varepsilon \neq (10)^{\infty}$ and
           $\varepsilon \preceq 1(10)^{\infty}$, then there exists a
           sequence $\varepsilon'$ also in $\Gamma$ such that
           $\varepsilon = \mu(\varepsilon')$.

\item [(d)] If $\varepsilon$ is a periodic sequence in $\Gamma$ such that
           $\varepsilon \preceq 1(10)^{\infty}$, then its smallest period is even.
\end{itemize}
\end{lemma}

\begin{proof}
We first prove (a). Define, for $j \geq 0$, the {\it block}
(or {\it word}) $z_j$ by
$$
z_j := (\varepsilon_{jd+1} \ \varepsilon_{jd+2} \cdots \varepsilon_{(j+1)d})
$$
so that we can write
$$
\varepsilon = (\varepsilon_1 \ \varepsilon_2 \cdots \varepsilon_d) \
(\varepsilon_{d+1} \ \varepsilon_{d+2} \cdots \varepsilon_{2d})
\cdots
(\varepsilon_{jd+1} \ \varepsilon_{jd+2} \cdots \varepsilon_{(j+1)d})
\cdots
= z_0 \ z_1 \ z_2 \cdots
$$
We prove by induction on $j \geq 0$ that $z_{2j} = z_0$ and
$z_{2j+1} = \overline{z_0}$. The case $j=0$ is exactly the
hypothesis in (a). Suppose the result is true for some $j \geq 0$,
i.e., that
$$
\varepsilon = (z_0 \ \overline{z_0})^{j+1} \ z_{2j+2} \ z_{2j+3} \cdots
$$
Now
\begin{equation}\label{eq1}
z_{2j+2} \ z_{2j+3} \cdots = \sigma^{2d(j+1)}(\varepsilon)
\preceq \varepsilon = z_0 \ \overline{z_0} \cdots
\end{equation}
hence $z_{2j+2} \preceq z_0$, and
$$
\overline{z_0} \ z_{2j+2} \ z_{2j+3} \cdots =
z_{2j+1} \ z_{2j+2} \ z_{2j+3} \cdots = \sigma^{d(2j+1)}(\varepsilon)
 \succeq \overline{\varepsilon} = (\overline{z_0} \ z_0)^{j+1} \cdots
$$
hence $\overline{z_0} \ z_{2j+2} \succeq \overline{z_0} \ z_0$,
which gives $z_{2j+2} \succeq  z_0$, hence $z_{2j+2} = z_0$. But the
inequality~(\ref{eq1}) now implies $z_{2j+3} \preceq
\overline{z_0}$. On the other hand
$$
z_{2j+3} \ z_{2j+4} \cdots = \sigma^{d(2j+3)}(\varepsilon) \succeq
\overline{\varepsilon} = \overline{z_0} \cdots
$$
hence $z_{2j+3} \succeq \overline{z_0}$ which finally implies that
$z_{2j+3} = \overline{z_0}$.

\bigskip

We now prove (b). Suppose that $\varepsilon$ belongs to $\Gamma$.
Since $\overline{\varepsilon} \preceq \varepsilon$, we have
$\varepsilon_1 = 1$. Applying Lemma~\ref{elementary-gamma}~(b) and
(d) to the inequalities $\overline{\varepsilon} \preceq
\sigma^k(\varepsilon) \preceq \varepsilon$, we get for all $k \geq
0$
$$
\sigma(\overline{\mu(\varepsilon)}) =
\sigma(\mu(\overline{\varepsilon})) \preceq
\sigma(\mu(\sigma^k(\varepsilon))) \preceq
\sigma(\mu(\varepsilon)).
$$
Hence, from Lemma~\ref{elementary-gamma}~(a),
\begin{equation}\label{eq2}
\sigma(\overline{\mu(\varepsilon)}) \preceq
\sigma^{2k+1}(\mu(\varepsilon)) \preceq
\sigma(\mu(\varepsilon)).
\end{equation}
Since $(\mu(\varepsilon))_1 = 1$, we have $\sigma(\mu(\varepsilon))
\preceq \mu(\varepsilon)$ and $\overline{\mu(\varepsilon)} \preceq
\sigma(\overline{\mu(\varepsilon)})$. Hence the above inequalities
yield
$$
\overline{\mu(\varepsilon)} \preceq
\sigma^{2k+1}(\mu(\varepsilon)) \preceq
\mu(\varepsilon).
$$
It remains to prove that for every $k \geq 0$
$$
\overline{\mu(\varepsilon)} \preceq
\sigma^{2k}(\mu(\varepsilon)) \preceq
\mu(\varepsilon).
$$
If $(\sigma^{2k}(\mu(\varepsilon)))_1 = 0$, then
$\sigma^{2k}(\mu(\varepsilon)) \prec \mu(\varepsilon)$ as
$\mu(\varepsilon)_1 = 1$. On the other hand, (\ref{eq2}) implies
that $\sigma(\overline{\mu(\varepsilon)}) \preceq \sigma
(\sigma^{2k}(\mu(\varepsilon)))$. Hence $\overline{\mu(\varepsilon)}
\preceq \sigma^{2k}(\mu(\varepsilon))$, since
$(\overline{\mu(\varepsilon)})_1 = (\sigma^{2k}(\mu(\varepsilon)))_1
\ (=0)$.

\medskip

\noindent If $(\sigma^{2k}(\mu(\varepsilon)))_1 = 1$, then
$\overline{\mu(\varepsilon)} \prec \sigma^{2k}(\mu(\varepsilon))$ as
$\overline{\mu(\varepsilon)}_1 = 0$. On the other hand, (\ref{eq2})
implies that $\sigma(\sigma^{2k}(\mu(\varepsilon))) \preceq
\sigma(\mu(\varepsilon))$. Hence $\sigma^{2k}(\mu(\varepsilon))
\preceq \mu(\varepsilon)$, since $(\sigma^{2k}(\mu(\varepsilon)))_1
= (\mu(\varepsilon))_1 \ (=1)$.

\bigskip

Now suppose that $\mu(\varepsilon)$ belongs to $\Gamma$, and that
$\varepsilon_1=1$. We clearly have $\overline{\varepsilon} \prec
\varepsilon$. It thus suffices to prove that, for any $k \geq 1$,
both inequalities $\sigma^k(\varepsilon) \preceq \varepsilon$ and
$\overline{\varepsilon} \preceq \sigma^k(\varepsilon)$ hold.

\begin{itemize}
\item
Let us prove the first inequality. Let $k \geq 1$. If
$\varepsilon_{k+1} = 0$, we have $\sigma^k(\varepsilon) =
\varepsilon_{k+1} \varepsilon_{k+2} \cdots \prec \varepsilon$. If
$\varepsilon_{k+1} = 1$, then either $\varepsilon_j = 1$ for all $j
\in [1, k+1]$ and $\sigma^k(\varepsilon) \preceq \varepsilon$ since
$\sigma^k(\varepsilon)$ begins with less $1$'s than $\varepsilon$,
or there exists $\ell \in [2, k+1]$ such that $\varepsilon_j = 1$
for all $j \in [\ell, k+1]$ and $\varepsilon_{\ell-1} = 0$. But then
$$
1 \ \varepsilon_{\ell} \ \overline{\varepsilon_{\ell}} \
\varepsilon_{\ell+1} \ \overline{\varepsilon_{\ell+1}} \cdots =
\overline{\varepsilon_{\ell-1}} \ \varepsilon_{\ell} \
\overline{\varepsilon_{\ell}} \ \varepsilon_{\ell+1} \
\overline{\varepsilon_{\ell+1}} \cdots = \sigma^{2\ell-2}(\mu(\varepsilon))
\preceq \mu(\varepsilon) = 1 \ \varepsilon_1 \ \overline{\varepsilon_1} \
\varepsilon_2 \ \overline{\varepsilon_2} \cdots
$$
hence
$$
\varepsilon_{\ell} \ \overline{\varepsilon_{\ell}} \
\varepsilon_{\ell+1} \ \overline{\varepsilon_{\ell+1}} \cdots \preceq
\varepsilon_1 \ \overline{\varepsilon_1} \
\varepsilon_2 \ \overline{\varepsilon_2} \cdots
$$
which easily implies
$$
\sigma^{\ell-1}(\varepsilon) = \varepsilon_{\ell} \ \varepsilon_{\ell+1}
\cdots \preceq \varepsilon_1 \ \varepsilon_2 \cdots = \varepsilon.
$$
This in turn implies $\sigma^k(\varepsilon) \preceq \varepsilon$,
since the sequence $\sigma^k(\varepsilon)$, beginning with less
$1$'s than $\sigma^{\ell-1}(\varepsilon)$, is smaller than
$\sigma^{\ell-1}(\varepsilon)$.

\item
Let us prove the second inequality. Let $k \geq 1$. If
$\varepsilon_{k+1} = 1$, we have $\sigma^k(\varepsilon) =
\varepsilon_{k+1} \varepsilon_{k+2} \cdots \succ
\overline{\varepsilon}$. If $\varepsilon_{k+1} = 0$, then there
exists $\ell \in [2, k+1]$ such that $\varepsilon_j = 0$ for all $j
\in [\ell, k+1]$ and $\varepsilon_{\ell-1} = 1$ (remember that
$\varepsilon_1 = 1$). But then
\[
0 \ \varepsilon_{\ell} \ \overline{\varepsilon_{\ell}} \
\varepsilon_{\ell+1} \ \overline{\varepsilon_{\ell+1}} \cdots =
\overline{\varepsilon_{\ell-1}} \ \varepsilon_{\ell} \
\overline{\varepsilon_{\ell}} \ \varepsilon_{\ell+1} \
\overline{\varepsilon_{\ell+1}} \cdots =
\sigma^{2\ell-2}(\mu(\varepsilon)) \succeq
\overline{\mu(\varepsilon)} = 0 \ \overline{\varepsilon_1} \
\varepsilon_1 \ \overline{\varepsilon_2} \ \varepsilon_2 \cdots
\]
hence
$$
\varepsilon_{\ell} \ \overline{\varepsilon_{\ell}} \
\varepsilon_{\ell+1} \ \overline{\varepsilon_{\ell+1}} \cdots \succeq
\overline{\varepsilon_1} \ \varepsilon_1 \
\overline{\varepsilon_2} \ \varepsilon_2 \cdots
$$
which easily implies
$$
\sigma^{\ell-1}(\varepsilon) = \varepsilon_{\ell} \
\varepsilon_{\ell+1} \cdots \succeq \overline{\varepsilon_1} \
\overline{\varepsilon_2} \cdots = \overline{\varepsilon}.
$$
This in turn implies $\sigma^k(\varepsilon) \succeq
\overline{\varepsilon}$, since the sequence $\sigma^k(\varepsilon)$,
beginning with less $0$'s than $\sigma^{\ell-1}(\varepsilon)$, is
larger than $\sigma^{\ell-1}(\varepsilon)$.

\end{itemize}

\bigskip

To finish with the proof of Lemma~\ref{other-gamma} it suffices to
prove (c): namely (d) is a consequence of (c) and of
Lemma~\ref{elementary-gamma}~(c). So, let $\varepsilon$ be a
sequence in $\Gamma$ with $\varepsilon \preceq 1(10)^{\infty}$.
Since $1(10)^{\infty} = \mu(1^{\infty})$, we may suppose that
$\varepsilon \prec 1(10)^{\infty}$. We thus have
$$
0(01)^{\infty} \prec \overline{\varepsilon} \preceq \sigma^k(\varepsilon)
\preceq \varepsilon \prec 1(10)^{\infty}
$$
for all $k \geq 0$. This implies in particular that $\varepsilon$
cannot contain three consecutive $1$'s nor three consecutive $0$'s.
The sequence $\varepsilon$ must begin with $1$. If $\varepsilon = 10
\cdots$, then $\varepsilon = (10)^{\infty}$ from
Lemma~\ref{other-gamma}~(a), which is excluded. Thus $\varepsilon =
11\cdots$, hence $\varepsilon = 110 \cdots$. From the inequality
$\varepsilon \prec 1(10)^{\infty}$, there is a maximal integer $i_1
\geq 1$ such that $\varepsilon = 1(10)^{i_1}\cdots$. Hence
$\varepsilon = 1(10)^{i_1}01\cdots$ (recall that $\varepsilon $ does
not contain three consecutive $0$'s). Then there exists a maximal
integer $j_1 \geq 1$ such that $\varepsilon = 1(10)^{i_1}(01)^{j_1}
\cdots$. But $0(01)^{j_1} \cdots = \sigma^{2i_1} \varepsilon \succeq
\overline{\varepsilon}= 0(01)^{i_1}(10)^{j_1} \cdots$. This implies
$j_1 \leq i_1$ and the next term in $\sigma^{2i_1} \varepsilon$ must
be $1$, i.e., $\sigma^{2i_1} \varepsilon = 0(01)^{j_1} 1 \cdots$.
Finally $\varepsilon = 1(10)^{i_1}(01)^{j_1} 1 \cdots$, with $0 \leq
j_1 \leq i_1$. A similar reasoning can be applied to
$\sigma^{2i_1+2j_1} \varepsilon = 11 \cdots$, yielding
$\sigma^{2i_1+2j_1} \varepsilon =1(10)^{i_2}(01)^{j_2}1\cdots$ with
$1 \leq i_2 \leq i_1$ and $0 \leq j_2 \leq i_1$. Thus, iterating, we
get
$$
\varepsilon =
1(10)^{i_1}(01)^{j_1} (10)^{i_2}(01)^{j_2} (10)^{i_3}(01)^{j_3} \cdots
$$
with $1 \leq i_k \leq i_1$ and $0 \leq j_k \leq i_1$. This implies $$
\varepsilon = \mu(1^{i_1} 0^{j_1} 1^{i_2} 0^{j_2} 1^{i_3} 0^{j_3} \cdots).
$$
The sequence $1^{i_1} 0^{j_1} 1^{i_2} 0^{j_2} 1^{i_3} 0^{j_3} \cdots$ belongs
to $\Gamma$ from Lemma~\ref{other-gamma}~(b).
\end{proof}

\begin{rmk}
In Lemma~\ref{other-gamma} Part~(a) is \cite[Lemme~2~b,
p.~27]{A-etat}. The ``only if'' part of (b) is on Page 26-06 in
\cite{A-bordeaux}. The proof of Part~(c) is inspired by the proof of
the partly more general ``Lemme~1'' on page 47 of \cite{A-etat}.
\end{rmk}

We are now ready to present and prove the following construction of
the sequences ($a_k)$.

\begin{prop}\label{smallest}
Denote by $a_k$ the smallest periodic sequence belonging to $\Gamma$
whose smallest period is equal to $k \geq 1$. Let $k=2^n(2m+1)$,
then

$$
a_k =  \left\{
\begin{array}{ll}
1^{\infty} &\mbox{\rm if } m=n=0 \ \mbox{\rm (i.e., $k=1$)}, \\
\mu^n(0^{\infty}) &\mbox{\rm if $m = 0$ and $n \geq 1$}, \\
\mu^n((1(10)^m)^{\infty}) \ &\mbox{\rm if } m \geq 1.
\end{array}
\right.
$$

\end{prop}

\begin{proof}

\begin{itemize}

\item The case $m=n=0$ is trivial.

\item Let us address the case $m=0$. For $n=1$, we have $k=2$, and it is
easy to see that $a_2 = (10)^{\infty} = \mu(0^{\infty})$. Suppose that
$a_{2^n} = \mu^n(0^{\infty})$ for some $n \geq 1$. Then $\mu(a_{2^n})$
has smallest period $2^{n+1}$, hence $a_{2^{n+1}} \preceq \mu(a_{2^n}) =
\mu^{n+1}(0^{\infty})$. Now $0^{\infty} \prec L \prec 1(10)^{\infty}$, hence
$a_{2^{n+1}} \preceq \mu^{n+1}(0^{\infty}) \prec \mu^{n+1}(L) = L
\prec 1(10)^{\infty}$.  This implies, from Lemma~\ref{other-gamma}(c), the
existence of a sequence $z$ in $\Sigma$ such that $a_{2^{n+1}} = \mu(z)$.
We have that $z$ is periodic, and its smallest period is $2^n$. Furthermore
$\mu(z) \preceq \mu^{n+1}(0^{\infty})$, hence $z \preceq \mu^n(0^{\infty})$.
This forces $z = \mu^n(0^{\infty})$, thus $a_{2^{n+1}} = \mu(z) =
\mu^{n+1}(0^{\infty})$.

\item Let us prove the result for $m \geq 1$ by induction on $n \geq 0$.
Take first $n=0$. Let $\varepsilon$ be a sequence in $\Gamma$ with smallest
period $(2m+1)$ for some $m \geq 1$. From Lemma~\ref{other-gamma}~(d), we
must have $\varepsilon \succ 1(10)^{\infty}$. Hence the prefix of $\varepsilon$
of length $2m+1$ must be larger than or equal to the prefix of $1(10)^{\infty}$
of length $2m+1$, i.e., $1(10)^m$. The sequence $\varepsilon$ being periodic with
smallest period $(2m+1)$, this implies that $\varepsilon \succeq (1(10)^m)^{\infty}$.
But this last sequence clearly belongs to $\Gamma$, which implies that it is the
smallest sequence in $\Gamma$ that has smallest period $2m+1$.

\item
Now, suppose the result is true for some $n \geq 0$. Let
$\varepsilon$ be the smallest sequence in $\Gamma$ whose smallest
period is $2^{n+1}(2m+1)$. Since the sequence
$\mu^{n+1}((1(10)^m)^{\infty})$ is in $\Gamma$ (use
Lemma~\ref{other-gamma}~(b)) and has smallest period
$2^{n+1}(2m+1)$ (use Lemma~\ref{elementary-gamma}~(c); note that we
{\it need} $m \neq 0$), we have $\varepsilon \preceq
\mu^{n+1}((1(10)^m)^{\infty})$. Since this clearly implies
$\varepsilon \preceq 1(10)^{\infty}$, Lemma~\ref{other-gamma}~(c)
gives the existence of a sequence $\varepsilon'$ in $\Gamma$ such
that $\varepsilon = \mu(\varepsilon')$. The inequality
$\mu(\varepsilon') = \varepsilon \preceq
\mu^{n+1}((1(10)^m)^{\infty})$ implies, using
Lemma~\ref{elementary-gamma}~(d), that $\varepsilon' \preceq
\mu^n((1(10)^m)^{\infty})$. But $\varepsilon'$ belongs to $\Gamma$
and has smallest period $2^n(2m+1)$ (use
Lemma~\ref{elementary-gamma}~(c)). Hence, from the induction
hypothesis, $\varepsilon' = \mu^n((1(10)^m)^{\infty})$. Thus,
$\varepsilon = \mu(\varepsilon') = \mu^{n+1}((1(10)^m)^{\infty})$.

\end{itemize}
\end{proof}

We conclude the section with an explicit formula for the $a_k$ via
fragments of the Thue-Morse sequence.

\begin{prop}\label{prop:explicit}

The sequences $(a_k)$ are related to the Thue-Morse sequence as follows.
Let $k = 2^n(2m+1)$. Then
$$
a_k =  \left\{
\begin{array}{ll}
\m_1^{\infty} &\mbox{\rm if } m=n=0 \ \mbox{\rm (i.e., $k=1$)}, \\
(\m_1 \m_2 \m_3 \cdots \m_{2^n-1} \overline{\m_{2^n}})^{\infty} \
&\mbox{\rm if $m = 0$ and $n \geq 1$}, \\
(\m_1 \m_2 \m_3 \cdots \m_{3\cdot2^n}(\m_1 \cdots
\overline{\m_{2^{n+1}}})^{m-1})^{\infty} &\mbox{\rm if } m \geq 1.
\end{array}
\right.
$$
\end{prop}

\begin{proof}
We may assume that $(m,n) \neq (0,0)$. We then note that (see the
proof of Lemma~\ref{elementary-gamma}~(e)), for any sequence
$\varepsilon$ in $\Sigma$ and for any $n \geq 0$, we have
$0\mu^n(\varepsilon) = \varphi^n(0\varepsilon)$. Hence to prove the
proposition, it suffices to show that, for any $n \geq 0$,
$$
\begin{array}{lll}
\varphi^n(0^{\infty}) &=&
0(\m_1 \m_2 \m_3 \cdots \m_{2^n-1} \overline{\m_{2^n}})^{\infty}, \\
\varphi^n(0(1(10)^m)^{\infty}) &=&
0(\m_1 \m_2 \m_3 \cdots \m_{3\cdot2^n}
(\m_1 \cdots \overline{\m_{2^{n+1}}})^{m-1})^{\infty}
\ \ \forall m \geq 1.
\end{array}
$$
Remembering that $\m_n$ is nothing but the parity of the sum of the
binary digits of $n$, we clearly have $\m_{2^n} = 1$, $0(\m_1 \m_2
\m_3 \cdots \m_{2^n-1} \overline{\m_{2^n}})^{\infty} = (\m_0 \m_2
\m_3 \cdots \m_{2^n-1})^{\infty}$, and
$$
\begin{array}{lll}
\m_1 \m_2 \m_3 \cdots \m_{3\cdot2^n} &=& \m_1 \m_2 \m_3 \cdots \m_{2^{n+1}}
\m_{2^{n+1}+1} \cdots \m_{2^{n+1}+2^n} \\
&=& \m_1 \m_2 \m_3 \cdots \m_{2^{n+1}-1} \overline{\m_0 \m_1 \cdots \m_{2^n-1}}
\ 0
\end{array}
$$
Denoting by $W_n$ the word (of length $2^n$) $W_n := \m_0 \m_1
\cdots \m_{2^n-1}$, what we have to prove boils down to
$$
\begin{array}{lll}
\varphi^n(0^{\infty}) &=& W_n^{\infty} \\
\varphi^n(0(1(10)^m)^{\infty}) &=&
(W_{n+1} \ \overline{W_n} \ W_{n+1}^{m-1})^{\infty}
\ \ \forall m \geq 1.
\end{array}
$$
Using the easily proven relation $W_{n+1} = \varphi(W_n)$ we have
$W_n = \varphi^n(0)$, $\overline{W_n} = \varphi^n(1)$, and $W_{n+1}
= W_n \overline{W_n}$, thus
$$
\varphi^n(0^{\infty}) = (\varphi^n(0))^{\infty} = W_n^{\infty}
$$
and
$$
\begin{array}{lll}
\varphi^n(0(1(10)^m)^{\infty}) &=&
\varphi^n(0)\varphi^n[(1(10)^m)^{\infty}]
= \varphi^n(0)[\varphi^n(1)(\varphi^n(10))^m]^{\infty} \\
&=& \varphi^n(0)[\varphi^n(1)(\varphi^n(1)\varphi^n(0))^m]^{\infty} \\
&=& W_n [\overline{W_n} \ (\overline{W_n} W_n)^m]^{\infty}
= W_n [\overline{W_n} \ (\overline{W_n} W_n)^{m-1} \overline{W_n} W_n]^{\infty} \\
&=& (W_n \overline{W_n} \ (\overline{W_n} W_n)^{m-1} \overline{W_n})^{\infty}
= (W_n \overline{W_n} \ \overline{W_n} (W_n \overline{W_n})^{m-1})^{\infty} \\
&=& (W_{n+1} \overline{W_n} W_{n+1}^{m-1})^{\infty}.
\end{array}
$$
\end{proof}

For the table of the first 8 values of $\be_n$ see Table~\ref{table}
below.

\begin{table}[ht]
\centerline{
\begin{tabular}{c|l|l|c|c}
       $\be_n$ & $d(\be_n)$ & minimal polynomial & numerical value &
       below $\be_{KL}$?\\
\hline $n=2$ & 11 & $x^2-x-1$ & 1.61803 & yes\\
\hline $n=3$ & 111 & $x^3-x^2-x-1$ & 1.83929 & no\\
\hline $n=4$ & 1101 & $x^3-2x^2+x-1$ & 1.75488 & yes\\
\hline $n=5$ & 11011 & $x^5-x^4-x^3-x-1$ & 1.81240 & no\\
\hline $n=6$ & 110101 & $x^6-x^5-x^4-x^2-1$ & 1.78854 & no \\
\hline $n=7$ & 1101011 & $x^6-2x^5+x^4-x^3-1$ & 1.80509 & no \\
\hline $n=8$ & 11010011 & $x^5-2x^4+x^2-1$ & 1.78460 & yes \\
\end{tabular}
}

\vskip0.5truecm

\caption{The table of $\beta_n$ for small values of $n$}
\label{table}
\end{table}

\section{Impossibility of continuous extension of
$F_\be$}\label{sec:extension}

Recall that the map $F_\be$ acts on a nowhere dense subset of
$I_\be$. It would be tempting to try to explain the Sharkovski\u{\i}
order in our model via some extension of $F_\be$ to the whole
interval and then applying the classical Sharkovski\u{\i} theorem to
that extended map. In this section we show that this is in fact
impossible.

\begin{thm}\label{thm:no-ext}
Let, as above, $\be_4\approx1.75488$ denote the unique root of
$x^3=2x^2-x+1$ lying in $(1,2)$. Assume we have $\be\in(\be_4,2)$;
then any continuous map $S_\be:I_\be\to I_\be$ such that
$S_\be|_{X_\be}=F_\be$ has a $k$-cycle for all $k\in {\mathbb{N}}$
provided $S_\be$ is monotonic on $[0,1/\be]$.
\end{thm}

\begin{rmk}In particular, this means that
any map of the form
\[
S_\be(x)=\begin{cases} \be x,& 0\le x < \frac1{\be},\\
G(x),& \frac1{\be} \le x \le \frac1{\be(\be-1)}\\
\be x -1, & \frac1{\be(\be-1)} < x \le \frac1{\be-1},
\end{cases}
\]
where $G$ is continuous, and $G\bigl(\frac1{\be}\bigr)=1,\
G\bigl(\frac1{\be(\be-1)}\bigr)=\frac{2-\be}{\be-1}$, has cycles of
any length provided $\be>\be_4$ -- see Figs below. On the other
hand, as we know, if $\be<\be_8$, then $F_\be$ itself has only
cycles of length~2 and 4. This means that there is no immediate
connection between the classical Sharkovski\u{\i} theorem and our
Theorem~\ref{thm:main}. (See Section~\ref{sec:trapezoidal} for a
less immediate connection.)

\begin{figure}
\centering \scalebox{0.6} {\includegraphics{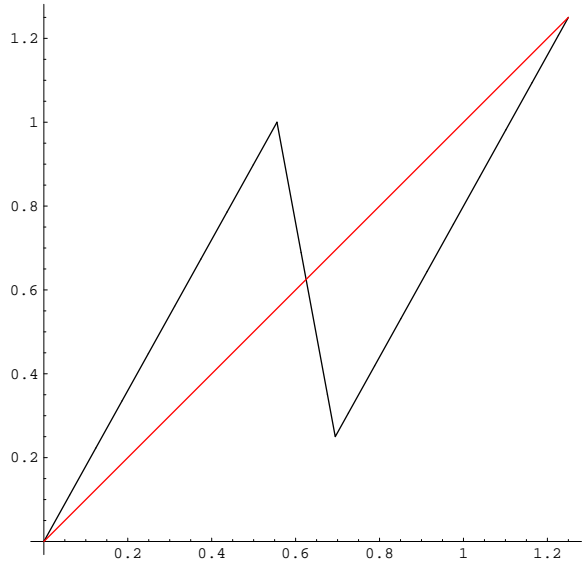}} \caption{A
continuous extension of $F_\be$ for $\be=1.8.$}\label{fig:sg1}
\end{figure}

\begin{figure}
\centering \scalebox{0.6} {\includegraphics{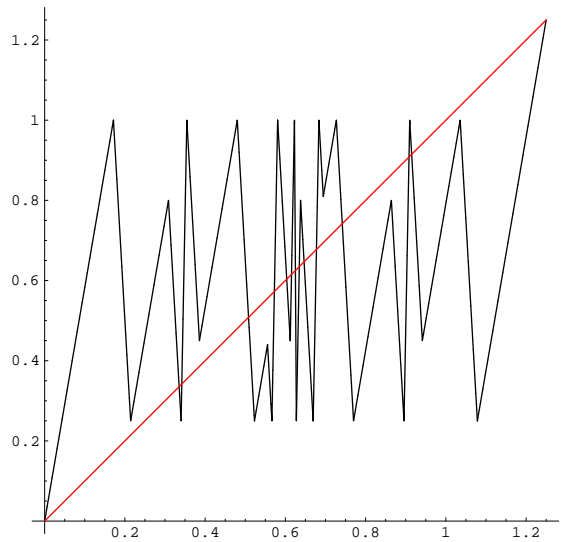}}
\caption{\dots and its third iterate -- observe all those
``parasite'' fixed points!}\label{fig:sg2}
\end{figure}

\end{rmk}

\begin{proof} Since $\be>\be_4$, we have the following 4-cycle in
$\Sigma_\be$:
\begin{align*}
x_1 &\sim (0011)^\infty,\,\,
x_2 \sim (0110)^\infty \\
x_3 &\sim (1100)^\infty,\,\, x_4 \sim (1001)^\infty
\end{align*}
Consequently, we have a 4-cycle for $F_\be$ which we denote by
$\{x_1,x_2,x_3,x_4\}$ as well. Notice that $x_1$ and $x_2$ lie in
the left hand side interval, while $x_3$ and $x_4$ belong to the
right hand side one.

Suppose such a map $S_\be$ exists; then we have
$S_\be^3(x_1)=x_4>x_1,\ S_\be^3(x_2)=x_1<x_2$. By the mean value
theorem, the map $S_\be^3$ has a fixed point $x_*$ between $x_1$ and
$x_2$. Since $F_\be(x_1)>x_1$ and $F_\be(x_2)>x_2$ and our
assumption on the monotonicity of $S_\be$ between $x_1$ and $x_2$,
we conclude that $S_\be(x_*)>x_*$, whence $x_*$ is a period~$3$
point for $S_\be$. By the classical Sharkovski\u{\i} theorem, this
implies that $S_\be$ has cycles of all possible lengths for
$\be>\be_4$.
\end{proof}

\section{Another proof of our main theorem using Sharkovski\u{\i}'s
classical theorem for trapezoidal maps}\label{sec:trapezoidal}

In the previous section we established the fact that there is no
natural extension of $F_\be$ to the whole interval which preserves
the delicate structure of the Sharkovski\u{\i} ordering on the
periodic orbits of $F_\be$. In this section we modify $F_\be$ by
flipping the right branch, which leads to the family of {\em
trapezoidal maps} (see, e.g., \cite{Louck}) and links our result to
the classical Sharkovski\u{\i} theorem.

More precisely, define the map $T_\be:I_\be\to I_\be$ as follows:
\[
T_\be(x)=\begin{cases}
\be x, &0\le x < \frac1{\be}\\
1, &\frac1{\be} \le x \le \frac1{\be(\be-1)}\\
\frac{\be}{\be-1}-\be x, & \frac1{\be(\be-1)}< x\le \frac1{\be-1}.
\end{cases}
\]

Following the standard notation, we denote the corresponding
intervals by $L,C$ and $R$ respectively -- see
Fig.~\ref{fig:branching2}. (Here
$C=\bigl[\frac1{\be},\frac1{\be(\be-1)}\bigr]$.)  We will write the
itineraries of points under $T_\be$ using this notation.

\begin{figure}[t]
    \centering \unitlength=0.3mm
   \centerline{
\begin{picture}(340,400)
\put(8,3){0}
\put(8,226){1}
\put(139,3){$\frac1{\be}$}
\put(185,3){$\frac1{\be(\be-1)}$}
\put(312,3){$\frac1{\be-1}$}
\put(76,28){$L$}
\put(165,28){$C$}
\put(250,28){$R$}
\put(20,20){\line(1,0){300}}
\put(20,320){\line(1,0){300}}
\put(20,320){\line(0,-1){300}}
\put(320,320){\line(0,-1){300}}
\dottedline(20,20)(320,320)
\dottedline(143,229)(143,20)
\dottedline(197,229)(197,20)
\dottedline(20,229)(143,229)
\thicklines
\put(20,20){\line(10,17){123}}
\put(320,20){\line(-10,17){123}}
\put(143,229){\line(1,0){54}}
\end{picture}}
\caption{The trapezoidal map $T_\be$ for $\beta=1.7$}
    \label{fig:branching2}
\end{figure}
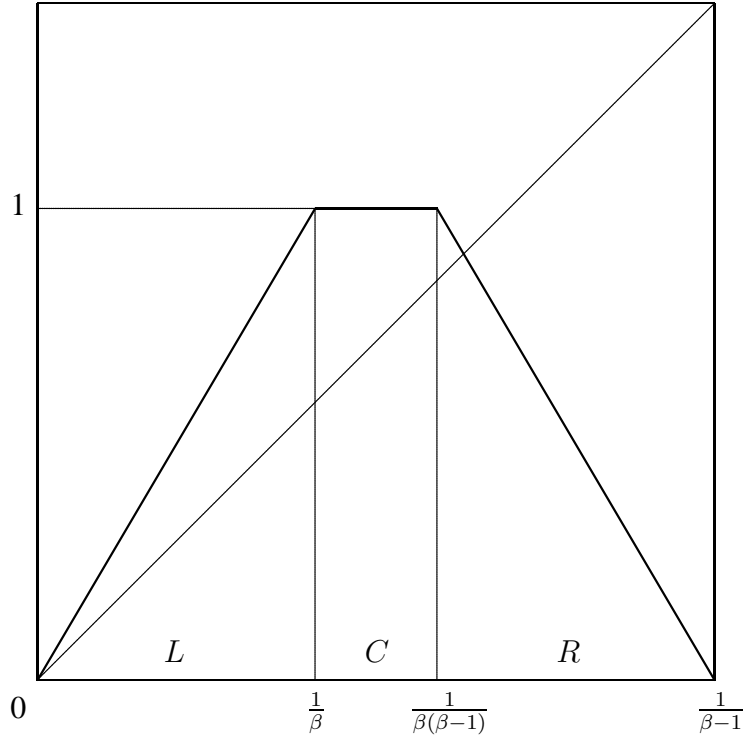

Our goal is to present another proof of the Sharkovski\u{\i} theorem
for the family $(\Si_\be,\si_\be)$ using the classical
Sharkovski\u{\i} theorem for $T_\be$.

Define the map $h:\Sigma\to\{L,R\}^{\mathbb N}$ as follows (from
here on $*$ denotes an arbitrary -- but fixed -- tail):
\begin{itemize}
\item $h(0*)=Lh(*)$;
\item $h(1^a0^b1*)=RL^{a-1}RL^{b-1}h(1*)$ for $a,b\ge1$;
\item $h(1^a0^\infty)=RL^{a-1}RL^\infty$;
\item $h(1^\infty)=RL^\infty$.
\end{itemize}

It is clear that $h$ is well defined and is one-to-one, with
$h^{-1}(RL^a RL^b)=1^{a+1}0^{b+1}$ for $a,b\ge0$ (on blocks). We
claim that $h$ in fact maps the orbits of $\si_\be$ into the orbits
of $T_\be$ which do not fall into $C$.

More precisely, put $\rho_\be:\{L,R\}\to I_\be$, where
$\rho_\be(\xi)=x$ such that $\xi$ is the itinerary of $x$ under $T_\be$.

\begin{lemma}
For $\e=1^\ell0*$ with $\ell\ge0$ we have
\[
F_\be^{\ell+1}(x)=T_\be^{\ell+1}(x),
\]
where $x=\rho_\be h(\e)$.
\end{lemma}

\begin{proof}Assume first that $\ell\ge1$; we have $h(\e)=RL^{\ell-1}Rh(*)$,
whence by the definition of $T_\be$,
\begin{align*}
T_\be(x)&=\frac{\be}{\be-1}-\be x,\\
T_\be^\ell(x)&=\frac{\be^\ell}{\be-1}-\be^\ell x,\\
T_\be^{\ell+1}(x)&=\frac{\be}{\be-1}-\frac{\be^{\ell+1}}{\be-1} +\be^{\ell+1}x\\
&=\be^{\ell+1}x-\be^\ell-\dots-\be^2-\be.
\end{align*}
It is easy to verify that the greedy $\be$-expansion of $x$ begins
with $1^\ell 0$, whence by the definition of $F_\be$ (see
Section~\ref{sec:intro}), $F_\be^{\ell+1}(x) =
\be^{\ell+1}x-\be^\ell-\dots-\be^2-\be= T_\be^{\ell+1}(x)$.

If $\ell=0$, then $F_\be(x)=T_\be(x)=\be x$.
\end{proof}

\begin{cor} For $\e=1^{a_1}0^{b_1}\dots 1^{a_s}0*$ with $a_j\ge1, b_j\ge1$, we have
\[
F_\be^{\ell+1}(x)=T_\be^{\ell+1}(x),
\]
where $x=\rho_\be h(\e)$ and $\ell=\sum_{j=1}^{s-1}(a_j+b_j)+a_s$.
\end{cor}

This result allows us to link the periodic orbits of $F_\be$ to
those of $T_\be$. Notice that since $h$ acts blockwise and does not
alter the length of a block, a $p$-periodic orbit of $F_\be$ maps
into a $q$-periodic orbit of $F_\be$, where $q$ divides $p$. In
fact, we will show that there are only two possibilities: either
$q=p$ or $q=p/2$ -- see below.

\begin{example} If $\e=(1100)^\infty$, then $h(\e)=(RL)^\infty$. A more
complicated example: $\e=(1101011\,\, 0010100)^\infty$ and
$h(\e)=(RLRRRRL)^\infty$. Notice that in both cases $\e=(v\ov
v)^\infty$ for some $v$. We will see later that this is always the
case when $h$ cuts a period in half.
\end{example}

Now we are ready to present an alternative proof of the main theorem
of this paper. Since the case of powers of $2$ is considered in
\cite{GS}, we assume $m\rhd k$ in the Sharkovski\u{\i} ordering and
$m$ is not a power of $2$. Suppose $F_\be$ has an $m$-cycle; then
$T_\be$ has an $m$-cycle or an $m/2$-cycle. In either case, by the
classical Sharkovski\u{\i} theorem, $T_\be$ has a $k$-cycle. We need
to make sure however that such a cycle does not involve $C$. Let us
call a cycle with this property an {\em L-R cycle} and prove a
version of the Sharkovski\u{\i} theorem\footnote{The authors are
grateful to Sebastian van Strien for the idea of the proof.}.

\begin{prop}\label{prop:lr}
If $T_\be$ has an L-R $m$-cycle, then it has an L-R $k$-cycle for
any $k\lhd m$.
\end{prop}
\begin{proof}
Let $\{x_1,\dots,x_m\}$ be the cycle in question. Without loss of
generality assume that $x_1$ is the point of this cycle closest to
$C$. If $x_1<1/\be$, we put
\[
\widetilde T_{\beta}(x)=\begin{cases}
\be x, & 0\le x< x_1,\\
\be x_1, & x_1\le x\le \frac1{\be-1}-x_1,\\
\frac{\be}{\be-1}-\be x,& \frac1{\be-1}-x_1 < x \le \frac1{\be-1}.
\end{cases}
\]
If $x_1>\frac1{\be(\be-1)}$, put
\[
\widetilde T_{\beta}(x)=\begin{cases}
\be x, & 0\le x<\frac1{\be-1}- x_1,\\
\frac{\be}{\be-1}-\be x_1, &\frac1{\be-1}-x_1\le x\le x_1,\\
\frac{\be}{\be-1}-\be x,& x_1 < x \le \frac1{\be-1}.
\end{cases}
\]
In other words, $\widetilde T_{\beta}$ is a trapezoidal map whose
graph is made out of the graph of $T_\be$ by ``sawing off'' its top
at the level $y=\be x_1$ or $y=\frac{\be}{\be-1}-\be x_1$
respectively. Notice that $\{x_1,\dots,x_m\}$ is still an $m$-cycle
for $\widetilde T_{\beta}$, whence, by the classical
Sharkovski\u{\i} theorem, $\widetilde T_{\beta}$ has a $k$-cycle
$\{y_1,\dots,y_k\}$. It suffices to observe that $y_j\notin C$ for
all $j$, because otherwise $\widetilde T_{\beta}(y_j)=\widetilde
T_{\beta}(x_1)$, and consequently, $\widetilde
T_{\beta}^r(y_j)\notin C$ for any $r\ge1$.

Hence, by our construction, $\{y_1,\dots,y_k\}$ is a sought L-R
cycle for $T_\be$.
\end{proof}

Thus, $T_{\beta}$ has an L-R $k$-cycle, whence by applying $h^{-1}$,
we conclude that $F_{\beta}$ (or $\sigma_{\beta}$) has a $k$-cycle
or an $\ell k$-cycle for some $\ell\ge2$. It suffices to discard the
latter case.

Let $\e = u^\infty$, with
\[
u=1^{a_1}0^{b_1}\dots 1^{a_{s-1}}0^{b_{s-1}}1^{a_s}\,\mid\,
0^{b_s}1^{a_{s+1}}0^{b_{s+1}}\dots 1^{a_r}0^{b_r},
\]
where $\mid$ separates the two halves of $u$. Then
\begin{equation}
\label{eq:hu}
h(u)=RL^{a_1-1}RL^{b_1-1}\dots RL^{a_s-1}\, \mid\,
RL^{b_s-1}RL^{a_{s+1}-1}RL^{b_{s+1}-1}\dots RL^{a_r-1}RL^{b_r-1}.
\end{equation}
From (\ref{eq:hu}) it is clear that if a word $u$ is not a power of
another word itself and $h(u)$ is a power, then it must be a square,
i.e., $\ell=2$.

Suppose $h(u)=ww$ for some $w$. Then
\[
RL^{a_1-1}RL^{b_1-1}\dots RL^{a_s-1}
= RL^{b_s-1}RL^{a_{s+1}-1}RL^{b_{s+1}-1}\dots RL^{a_r-1}RL^{b_r-1},
\]
i.e., $a_1=b_s,b_1=a_{s+1},\dots, a_s=b_r$. In other words,
$u=v\ov v$ for $v=1^{a_1}0^{b_1}\dots 1^{a_{s-1}}0^{b_{s-1}}1^{a_s}$.

Thus, if a sequence $\e$ is of smallest period~$2k$ and $h(\e)$ has
smallest period~$k$, then $\e=(v\ov v)^\infty$ for some $v$. Our
goal is to show that for such an $\e$ one can find
$\e'\in\Sigma_\be$ of smallest period~$k$, which will conclude the
proof.

Analogously to the proof of Proposition~\ref{main}, we consider all
the shifts of $\e$ and $\ov\e$ (which are all in $\Sigma_\be$) and
choose the maximal one. Hence without loss of generality, we may
again assume $\e\in\Gamma$, where $\Gamma$ is given by
(\ref{eq:gamma}).

\begin{prop}
Assume $\e\in\Gamma$ is a sequence of smallest period $2k$, of the
form $\e=(v\ov v)^\infty$ for some $v$, where $|v|=k>1$. Then there
exists a sequence $\e'\in\Gamma$ such that $\varepsilon'$ has
smallest period $k$ and $\e'\prec\e$.
\end{prop}

\begin{proof}
This result can be deduced from the combination of two results of
\cite{A-etat} (namely Proposition~2 on p.~34 applied to the sequence
$\varepsilon$ and Proposition~1 on p. 32), but we give a direct
proof. Since the sequence $\varepsilon$ begins with $1$, and since
$\sigma^{2k - 1}\varepsilon \prec \varepsilon$, the word
$\overline{v}$ must end in $0$, hence the word $v$ must end in $1$.
Let $v := w1$, hence $v\overline{v} = w 1 \overline{w} 0$.

The sequence $\varepsilon' := (w0)^{\infty}$ has period $k$. This is
the smallest period of the sequence $\varepsilon'$: if this were not
the case, we would have $w0 = (z0)^{\ell}$ for some word $z$ and
some integer $\ell \geq 2$. Thus $w = (z0)^{\ell - 1} z$. This would
imply $\varepsilon = (w 1 \overline{w} 0)^{\infty} = ((z0)^{\ell -
1} z 1(\overline{z} 1)^{\ell - 1} \overline{z} 0)^{\infty}$. But
then $\sigma^{2\ell-2} \varepsilon$ begins with $z1$ and the
condition $\sigma^{2\ell-2} \varepsilon \preceq \varepsilon$ would
not be satisfied.

We clearly have $\varepsilon' \prec \varepsilon$. To prove that
$\varepsilon'$ belongs to $\Gamma$, it clearly suffices to prove
that if the word $w$ is written as $w := xy$ (thus $(w0)^{\infty} =
(xy0)^{\infty})$, with the condition $0 < |x| < |w| = k-1$, then
\begin{equation}\label{eq:ineq1}
\overline{x} \ \overline{y} \ 1 \prec y \ 0 \ x \prec x \ y \ 0,
\end{equation}
thus yielding $\overline{\varepsilon'} \prec \sigma^{|x|}
\varepsilon' \prec \varepsilon'$ or $\overline{x} \ \overline{y} \ 1
= y \ 0 \ x \prec x \ y \ 0$, thus yielding $\overline{\varepsilon'}
= \sigma^{|x|} \varepsilon' \prec \varepsilon'$.

\medskip

Let us prove (\ref{eq:ineq1}). Since $\varepsilon = (w \ 1 \
\overline{w} \ 0)^{\infty} = (x \ y \ 1 \ \overline{x} \
\overline{y} \ 0)^{\infty}$ belongs to $\Gamma$, we have
$\sigma^{|x|}(\varepsilon) \preceq \varepsilon$, hence $y \ 1 \
\overline{x} \preceq x \ y \ 1$. Notice that $y \ 1 \ \overline{x}$
cannot be equal to $x \ y \ 1$. Namely, if these two words were
equal, this would imply $y \ 1 \ \overline{x} \ \overline{y} \ 0 \ x
= x \ y \ 1 \ \overline{x} \ \overline{y} \ 0$. This shows that the
words $x$ and $y \ 1 \ \overline{x} \ \overline{y} \ 0$ would
commute: from a theorem of Lyndon and Sch\"utzenberger \cite{LS}
this would imply that there exist a word $z$ and two positive
integers $a$ and $b$ such that $x = z^a$ and $y \ 1 \ \overline{x} \
\overline{y} \ 0 = z^b$. Hence $w \ 1 \ \overline{w} \ 0 = z^{a+b}$
and $2k$ would not be the smallest period of the sequence
$\varepsilon$. We thus have $y \ 1 \ \overline{x} \prec x \ y \ 1$,
whence
\begin{equation}\label{eq:ineq2}
y \ 1 \ \overline{x} \preceq x \ y \ 0.
\end{equation}
Obviously, $y \ 0 \ x \prec y \ 1 \ \overline{x}$, whence $y \ 0 \ x
\prec x \ y \ 0$, which proves the RHS inequality in
(\ref{eq:ineq1}).

\medskip

Let us prove that $\overline{x} \ \overline{y} \ 1 \prec y \ 0 \ x$,
or, equivalently, that
\begin{equation}\label{eq:ineq3}
\overline{y} \ 1 \ \overline{x} \prec x \ y \ 0.
\end{equation}
We can write $xy=YX$, with $|X| = |x|$ and $|Y| = |y|$. Thus,
$\varepsilon = (Y \ X \ 1 \ \overline{x} \ \overline{y} \
0)^{\infty}$. Since $\sigma^{2|x|+1+|y|} \varepsilon \preceq
\varepsilon$, we have $\overline{y} \preceq Y$. Now, if
$\overline{y} \prec Y$, then $\overline{y} \ 1 \ \overline{x} \prec
Y \ X \ 0 = x \ y \ 0$, which is the sought
inequality~(\ref{eq:ineq3}), and we are done.

If $\overline{y} = Y$, then $\varepsilon = (Y \ X \ 1 \ \overline{x}
\ Y \ 0)^{\infty}$. We claim that $X \ 1 \ \overline{x}$ cannot
begin with $0$, because if this were the case, say $X \ 1 \
\overline{x} := 0 \ t$, then $\varepsilon = (Y \ 0 \ t \ Y \
0)^{\infty}$. The inequality $\sigma^{|y|+|t|+1} \varepsilon \preceq
\varepsilon$ would imply that $Y \ 0 \ Y \ 0 \ t \preceq Y \ 0 \ t \
Y \ 0$, hence $Y \ 0 \ t \preceq t \ Y \ 0$. On the other hand
$\sigma^{|y|+1} \varepsilon \preceq \varepsilon$ implies that $t \ Y
\ 0 \preceq Y \ 0 \ t$. We would thus have $t \ Y \ 0 = Y \ 0 \ t$.
In other words, $Y0$ and $t$ commute, whence, as above, there exist
a word $z$ and two positive integers $a$ and $b$, with $Y0 = z^a$
and $t = z^b$. Consequently,  $\varepsilon = (z^{2a+b})^{\infty}$,
which contradicts the minimality of the period of the sequence
$\varepsilon$.

Therefore, $X \ 1 \ \overline{x}$ must begin with $1$, say, $X \ 1 \
\overline{x} := 1 \ t$. We have $\varepsilon = (Y \ 1 \ t \ Y \
0)^{\infty}$. Now $\overline{\varepsilon} \preceq \sigma^{|y|+1}
\varepsilon$ implies that $\overline{Y} \ 0 \ \overline{t} \preceq t
\ Y \ 0 \prec t \ Y \ 1$. Hence, in view of $\ov y=Y$ and $\ov X\ 0\
x = 0\ \ov t$, we have
\begin{equation}\label{eq:ineq4}
\begin{aligned}
\overline{y} \ 1 \ \overline{x} \ \overline{y} \ 0 \ x &=
Y \ 1 \ \overline{Y} \ \overline{X} \ 0 \ x =
Y \ 1 \ \overline{Y} \ 0 \ \overline{t}\\
& \prec Y \ 1 \ t \ Y \ 1 =
Y \ X \ 1 \ \overline{x} \ Y \ 1 \\
&= x \ y \ 1 \ \overline{x} \ Y \ 1.
\end{aligned}
\end{equation}
This, in turn, implies $\overline{y} \ 1 \ \overline{x} \preceq x \
y \ 1$. Notice that if we had $\overline{y} \ 1 \ \overline{x} = x \
y \ 1$, then (\ref{eq:ineq4}) would imply $\overline{y} \ 0 \ x
\prec \overline{x} \ Y \ 1 = \overline{x} \ \overline{y} \ 1$, i.e.,
by barring everything, $x \ y \ 0 \prec y \ 1 \ \overline{x}$ --
which clearly contradicts (\ref{eq:ineq2}).

Hence $\overline{y} \ 1 \ \overline{x} \prec x \ y \ 1$, and thus,
by (\ref{eq:ineq4}), $\overline{y} \ 1 \ \overline{x} \preceq x \ y
\ 0$. Either $\overline{y} \ 1 \ \overline{x} \prec x \ y \ 0$,
which is precisely the required LHS inequality in (\ref{eq:ineq1}),
and we are done, or $\overline{y} \ 1 \ \overline{x} = x \ y \ 0$,
i.e., $y \ 0 \ x = \overline{x} \ \overline{y} \ 1$. This implies
$\sigma^{|x|} \varepsilon' = (y \ 0 \ x)^{\infty} = (\overline{x} \
\overline{y} \ 1)^{\infty} = \overline{\varepsilon'}$, which,
together with the proven RHS of the inequality (\ref{eq:ineq1})
yields $\ov{\e'}\in\Gamma$.
\end{proof}

Thus, we have constructed a periodic sequence $\e'\in\Gamma$ with
smallest period $k$, which implies
$\sigma^j\e'\preceq\e'\prec\e\prec d(\be)$, and similarly,
$\sigma^j\ov{\e'}\prec d(\be)$. Hence $\e'\in\Sigma_\be$, and this
concludes the second proof of Theorem~\ref{thm:main} in the case
when $m$ is not a power of 2.

\section{Final remarks and an open problem}

\begin{rmk}
The case $m=2^n$ is a bit more delicate: here $h$ can map a periodic
sequence with smallest period $m$ into a periodic sequence with
smallest period $m/2$, for instance, $h((1100)^\infty)=(RL)^\infty$.
A direct inspection shows that a $2$-cycle for $T_\be$ does appear
at $\be=\be_2$, but it is $(RC)^\infty$, not $(RL)^\infty$. The
latter cycle in fact appears at $\be=\be_4$, where the former one
disappears.

\medskip
We conjecture that the first $2^n$-cycle to arise for $T_\be$
appears at $\be=\be_{2^n}$ but always involves $C$. The proof is
left to the interested reader.
\end{rmk}

\begin{rmk}
H.~Bruin is his Ph.~D.~Thesis \cite{bruin} proved various results on
the Sharkovski\u{\i} theorem for unimodal maps.
\end{rmk}

\begin{rmk}
Let, as above, $\tau_\be:[0,1)\to[0,1)$ denote the (greedy)
$\be$-transformation, i.e., $\tau_\be x=\be x\bmod 1$. Our remark
consists in a simple observation that there is no Sharkovski\u{\i}
theorem for the family $([0,1),\tau_\be)_{\be\in(1,2)}$. Indeed, the
set of admissible sequences here is
\begin{equation}\label{eq:greedy}
\widetilde\Sigma_\be=\{\e\in\Sigma: \sigma^j\e\prec
d(\be),\quad\forall j\ge0\}
\end{equation}
(see \cite{Pa}), and it is obvious that the smallest periodic
sequence $\e$ with smallest period $n$ such that $\e_1=1$ and
$\sigma^j\e\preceq\e$ for all $j\ge0$, is $(10^{n-1})^\infty$. Hence
the analogue of $U_n$ is $\widetilde U_n=(q_n,2)$, where $q_n$ is
the appropriate root of $x^n=x^{n-1}+1$, i.e.,
$d'(q_n)=(10^{n-1})^\infty$.

Hence $\widetilde U_n\subset\widetilde U_k$ iff $n<k$, which is not
a particularly interesting result. Comparing (\ref{eq:uniq}) and
(\ref{eq:greedy}), we see that the extra condition
$\sigma^j\ov\e\prec d(\be)$ makes all the difference.
\end{rmk}

\begin{rmk}
Let $\prec_u$ denote the {\em unimodal order} on the itineraries of
$T_\be$, i.e., $L\prec_u C\prec_u R$ and $\e\prec_u\e'$ if
$\e_i\equiv\e'_i,\ 1\le i\le k$ and either
$\e_{k+1}\prec_u\e_{k+1}'$ with $\#\{i\in[1,k] : \e_i=R\}$ even or
$\e_{k+1}\succ_u\e_{k+1}'$ with $\#\{i\in[1,k] : \e_i=R\}$ odd (see,
e.g., \cite{MT}).

\begin{prop}We have for $\e,\e'\in\Si$,
\[
\e\prec \e' \iff h(\e)\prec_u h(\e').
\]
\end{prop}

\begin{proof}
Essentially this claim can be found in \cite{Cosnard}, but for the
reader's convenience we will give a sketch of the proof. Let
\begin{align*}
\e&=1^{a_1}0^{b_1}\dots 1^{a_s}0^{b_s}0\dots,\\
\e'&=1^{a_1}0^{b_1}\dots 1^{a_s}0^{b_s}1\dots.
\end{align*}
Then
\begin{align*}
h(\e)&=RL^{a_1-1}RL^{b_1-1}\dots RL^{a_s-1}RL^{b_s-1}L\dots \\
h(\e')&=RL^{a_1-1}RL^{b_1-1}\dots RL^{a_s-1}RL^{b_s-1}R\dots,
\end{align*}
whence $h(\e)\prec_u h(\e')$. The other cases are similar.
\end{proof}
\end{rmk}

\medskip\noindent
{\bf Open problem.} Let $\Omega=\{\p_1,\dots,\p_m\}$ be points in
$\mathbb R^d$ and let $S_\Omega(\be)$ denote the set of
``$\be$-expansions'', where the ``digits'' are taken from the set
$\Omega$. More precisely, put, for $\be>1$,
\[
S_\Omega(\be)=\left\{(\be-1)\sum_{n=1}^\infty \be^{-n}{\bm a}_n \mid
\bm a_n \in \Omega\right\}.
\]
Clearly, $S_\Omega(\be)$ is a subset of the convex hull of $\Omega$,
and each $\bm x\in S_\Omega(\be)$ has at least one {\em address}
$(\bm a_1,\bm a_2,\dots)\in\Omega^{\mathbb N}$. Similarly to our
setting, one can define the set of points which have a unique
address and enquire about possible periods for such points.

In \cite[Section~4]{Sid3} the third author studied the case $d=2,
m=3$ (with noncollinear points $\p_1,\p_2,\p_3$) and showed that the
first period to appear is period~$3$. It is also easy to show that
the last period to appear is period~$2$, so we have a reverse
Sharkovski\u{\i} order here -- at least at the endpoints. Other
periods are much harder to deal with though, because of the holes in
$S_\Omega$, and it is not even clear whether $U_n$ in this model is
an interval for each $n$.

Obtaining a direct analogue of the Sharkovski\u{\i} theorem for the
shift on the set of unique addresses would be intriguing.

\medskip\noindent
{\bf Acknowledgement.} The authors are indebted to Henk Bruin, Paul
Glendinning and Sebastian van Strien for fruitful discussions and
general insight into one-dimensional continuous dynamics.


\begin{thebibliography}{99}

\bibitem{A-bordeaux} J.-P. Allouche, \textit{Une propri\'et\'e
extr\'emale de la suite de Thue-Morse en rapport avec les cascades
de Feigenbaum}, S\'em. Th\'eorie des Nombres Bordeaux (1981-1982),
26-01--26-17.

\bibitem{A-etat} J.-P. Allouche, \textit{Th\'eorie des Nombres et
Automates}, Th\`ese d'\'Etat, Universit\'e Bordeaux I, (1983).

\bibitem{AllCos} J.-P. Allouche and M. Cosnard, \textit{It\'erations
de fonctions unimodales et suites engendr\'ees par automates},
C.~R.~Acad.~Sci.~Paris, S\'erie I \textbf{296} (1983), 159--162.

\bibitem{AC} J.-P. Allouche and M. Cosnard, \textit{The
Komornik-Loreti constant is transcendental}, Amer. Math. Monthly
\textbf{107} (2000), 448--449.

\bibitem{AC2} J.-P. Allouche and M. Cosnard,
\textit{Non-integer bases, iteration of continuous real maps, and an
arithmetic self-similar set}, Acta Math. Hung. \textbf{91} (2001),
325--332.

\bibitem{AS} J.-P. Allouche and J. Shallit, \textit{The
ubiquitous Prouhet-Thue-Morse sequence}, in C. Ding. T. Helleseth,
and H. Niederreiter, eds., Sequences and Their Applications:
Proceedings of SETA '98, Springer-Verlag, 1999, pp. 1--16.


\bibitem{bruin}H. Bruin,
{\em Invariant measures of interval maps}, Ph.~D.~Thesis, Delft,
1994.

\bibitem{Cosnard} M. Cosnard,
\textit{\'Etude de la classification topologique des fonctions
unimodales}, Ann. Inst. Fourier \textbf{35} (1985), 59--77.

\bibitem{DarKat} Z. Dar\'oczy and I. Katai, \textit{Univoque
sequences}, Publ. Math. Debrecen \textbf{42} (1993), 397--407.

\bibitem{DK} Z. Dar\'oczy and I. K\'atai, \textit{On the structure of
univoque numbers}, Publ. Math. Debrecen \textbf{46} (1995), 385--408.

\bibitem{RD} R. Devaney, \textit{An Introduction to Chaotic Dynamical
Systems}, Addison-Wesley Publishing Company, 1989.

\bibitem{EJK} P. Erd\H os, I. Jo\'o, and V. Komornik,
\textit{Characterization of the unique expansions
$1=\sum_{i=1}^\infty q^{-n_i}$ and related problems}, Bull. Soc.
Math. Fr. \textbf{118} (1990), 377--390.

\bibitem{GS} P. Glendinning and N. Sidorov, \textit{Unique
representations of real numbers in non-integer bases}, Math. Res.
Letters \textbf{8} (2001), 535--543.

\bibitem{KL} V. Komornik and P. Loreti, \textit{Unique
developments in non-integer bases}, Amer. Math. Monthly \textbf{105}
(1998), 636--639.

\bibitem{KV} V. Komornik and M. de Vries,
\textit{Unique expansions of real numbers}, Preprint 2007. Available
at \newline {\tt http://arxiv.org/abs/math/0609708v3}

\bibitem{Louck}J.~D.~Louck and N. Metropolis,
\textit{Symbolic Dynamics of Trapezoidal Maps}, Mathematics and its
Applications, 27, D.~Reidel Publishing Co., Dordrecht, 1986.

\bibitem{LS} R.C. Lyndon and M. P. Sch\"utzenberger,
\textit{The equation $a^M = b^N c^P$ in a free group}, Michigan
Math. J. \textbf{9} (1962), 289--298.

\bibitem{MT} J. Milnor and W.~Thurston,
\textit{On iterated maps of the interval}. Dynamical systems
(College Park, MD, 1986--87), 465--563, Lecture Notes in Math.,
1342, Springer, Berlin, 1988.

\bibitem{Pa} W. Parry, \textit{On the $\be$-expansions of real
numbers}, Acta Math. Acad. Sci. Hung. \textbf{11} (1960), 401--416.

\bibitem{SH} O. M. Sharkovski\u{\i},
\textit{Co-existence of cycles of a continuous mapping of a line
onto itself.}, Ukranian Math. Z. \textbf{16} 1964, 61--71.

\bibitem{S} N. Sidorov,
\textit{Almost every number has a continuum of $\beta$-expansions},
Amer. Math. Monthly \textbf{110} (2003), 838--842.

\bibitem{Sid3}N. Sidorov,
\textit{Combinatorics of linear iterated function systems with
overlaps}, Nonlinearity {\bf 20} (2007) 1299--1312.

\end{thebibliography}
\end{document}